\documentclass{amsart}
\usepackage{amsmath,amsthm}
\usepackage{amsfonts,amssymb}

\newtheorem{theorem}{Theorem}[section]

\newtheorem{corollary}[theorem]{Corollary}
\theoremstyle{definition}

\theoremstyle{remark}
\newtheorem{remark}[theorem]{Remark}

\numberwithin{equation}{section}

\begin{document}
\UseRawInputEncoding
\title{Dunkl intertwining operator for symmetric groups}


\author{Hendrik De Bie}
\address{Department of Electronics and Information Systems \\Faculty of Engineering and Architecture\\Ghent University\\Krijgslaan 281, 9000 Gent\\ Belgium.}
\curraddr{}
\email{Hendrik.DeBie@UGent.be}
\thanks{}

\author{Pan Lian}
\address{School of Mathematical Sciences -- Tianjin Normal University\\
Binshui West Road 393, Tianjin 300387\\ P.R. China}
\curraddr{}
\email{panlian@tjnu.edu.cn}
\thanks{}


\date{}

\dedicatory{}


\begin{abstract}
In this note,  we express explicitly  the  Dunkl kernel and generalized Bessel functions of type $A_{n-1}$ by the Humbert's function $\Phi_{2}^{(n)}$,  with one variable specified.  The obtained formulas lead to a new proof of Xu's  integral expression  for the intertwining operator associated to  symmetric groups, which was recently reported in \cite{Xu}.
\end{abstract}

\maketitle
\section{Introduction}
Dunkl operators are a family of commuting differential-difference operators associated with a finite reflection group. They were introduced by Dunkl in the late eighties in \cite{du1}. During the last years, Dunkl operators have played an important role in generalizing classical Fourier analysis and have made a deep and lasting impact on  special functions theory. Furthermore, in the symmetric group case, Dunkl theory is naturally connected with the Schr\"odinger operators for Calogero-Sutherland type quantum many body systems, see e.g. \cite{JP}.

One of the important problems still open in this theory is to construct the explicit formulas for the intertwining operator (see Section 2.1) and the integral kernel of the  Dunkl transform \cite{ deJ,du2} for concrete finite groups.  Concrete formulas for the Dunkl kernel and intertwining operator are the premise to do a lot of hard analysis, see e.g. \cite{Xu}. This problem has received considerable attention in very recent times, especially in the case of dihedral groups and symmetric groups.

In the dihedral group setting, some explicit expressions for the Dunkl kernel were obtained in the last years, see \cite{DDY} for an explicit series expansion and \cite{CDL} for a complicated integral expression. Without integral expression, the action of the intertwining operator on polynomials has been studied in \cite{dudi} and \cite{Xu2}. In 2020, for a restricted class of functions, an elegant integral expression was obtained for the intertwining operator in \cite{Xu1}, using an integral over the simplex. This result was further explored in the general case in \cite{dl} and \cite{DD1}. In particular, a new integral expression for the intertwining operator was given in \cite{dl}.


For the symmetric groups, besides the explicit formula for the symmetric group $S_{3}$ determined by Dunkl long ago in \cite{du}, there are mainly two approaches to study this problem. The first one starts from constructing the explicit formulas of the generalized Bessel functions, then studies the Abel transform and its dual, which coincides with the intertwining operator for functions invariant under the finite reflection group. For example, complicated iterative formulas for the generalized Bessel functions  were given in \cite{b1,Ps} on a hyperplane of the Euclidean space $\mathbb{R}^{n}$. The other approach starts from determining the intertwining operator directly for a class of functions. For example, Dunkl himself determined the action of the intertwining operator on polynomials in \cite{du3}. Recently, an explicit integral expression of the intertwining operator for functions of single component was obtained by Xu in \cite{Xu}. Note that his results are proven by direct verification of the intertwining relations (\ref{ir1}). It was pointed out that the sets of functions considered in the first approach and in \cite{Xu} do not overlap.

The main contribution of the present paper is that we give an alternative proof of Xu's formula for the symmetric group starting from the generalized Bessel function. To do that, we first express the generalized Bessel function and the Dunkl kernel for a fixed variable by the Humbert  function $\Phi_{2}^{(n)}$. This is obtained using the limiting relations between the generalized Bessel function and the Heckman-Opdam hypergeometric function and the recent result in \cite{ST}.
This moreover provides the link between both approaches mentioned above.

 This note is organized as follows. In Section 2, we give the basic notions of Dunkl theory and the Humbert function. Section 3 is devote to the explicit formulas of the generalized Bessel function. In Section 4, we study the Dunkl kernel and the intertwining operator. We give a conclusion at the end of this note.
\section{Preliminaries}
\subsection{Dunkl operator and Dunkl's intertwiner}
The Dunkl operators associated to the symmetric group $S_{n}$ (or root system $A_{n-1}$) over $\mathbb{R}^{n}$ are defined by
\begin{eqnarray*}
D_{i}f(x)=\frac{\partial}{\partial x_{i}}f(x)+\kappa\sum_{j=1,j\neq i}^{n}\frac{f(x)-f(x(i,j) )}{x_{i}-x_{j}}, \quad 1\le i\le n,
\end{eqnarray*}
where $\kappa$ is a nonnegative real number and $(i,j)$ is the transposition exchanging  the $i$th and $j$th coordinates of $x\in \mathbb{R}^{n}$, see \cite{du1}.

Denote  $\mathcal{P}_{m}^{n}$ the space of homogeneous polynomial of degree $m$ in $n$ variables.
There exists a unique linear operator $V_{\kappa}: \mathcal{P}_{m}^{n}\rightarrow \mathcal{P}_{m}^{n}, $ called  intertwining operator \cite{du}, satisfying the  relations
\begin{eqnarray} \label{ir1}
D_{i}V_{\kappa}=V_{\kappa}\partial_{i}, \quad 1\le i \le n
\end{eqnarray}
and $V_{\kappa}1=1$.  It was abstractly proved that the intertwining operator can be expressed as an integral operator in \cite{RV}. More precisely,  there exists a nonnegative probability measure $d\mu_{x}$ such that \begin{eqnarray*}
V_{\kappa}f(x)=\int_{\mathbb{R}^{n}}f(y) d\mu_{x}(y).
\end{eqnarray*}
However, the explicit expression of the intertwining operator is only known for a few groups, e.g. $G=\mathbb{Z}_{2}^{n}$ and $S_{3}$, see \cite{du}. Recently, partial result for the intertwining operator associated to the dihedral groups  was obtained in \cite{Xu1} and a full expression was  obtained in \cite{dl}.

The Dunkl kernel is defined by
\begin{eqnarray*}
E_{\kappa}(x,y):= V_{\kappa}\left[e^{\langle \cdot, y\rangle}\right](x), \qquad x,y \in \mathbb{R}^{n}
\end{eqnarray*}
and is the integral kernel of the Dunkl transform \cite{du2, deJ}.
The symmetric analogue of the Dunkl kernel is called the generalized Bessel function. It is denoted by $J_{\kappa}(x,y)$ and given by
\begin{eqnarray} \label{bes1}
J_{\kappa}(x,y):= \frac{1}{n!}\sum_{\sigma\in S_{n}}E_{\kappa}(x, y\sigma),
\end{eqnarray}
in the case of the symmetric group.
Some complicated integral expressions for the generalized Bessel function of type $A_{n-1}$ were given in \cite{b1} and \cite{Ps}.

\subsection{Heckman-Opdam hypergeometric function and the asymptotic relationship}

Basics of the trigonometric Dunkl theory can be found in the review \cite{JP}.
The Cherednik operator $T_{\xi}$, $\xi \in \mathbb{R}^{n}$ associated with the root system $R$  and the non-negative multiplicity function $\kappa$ is defined by
\[ T_{\xi} f(x)=\partial_{\xi} f(x)+\sum_{\alpha\in R_{+}}\kappa_{\alpha}
\langle \alpha, \xi\rangle \frac{f(x)-f(r_{\alpha}(x))}{1-e^{-\langle\alpha, x \rangle} }-\langle \rho(\kappa), \xi \rangle f(x),   \]
where $\rho(\kappa)=\frac{1}{2}\sum_{\alpha\in R_{+}}\kappa_{\alpha} \alpha $ and $r_\alpha$ the reflection in the hyperplane orthogonal to $\alpha$. The Weyl group for the root system $R$ is denoted by $W$.

The Heckman-Opdam hypergeometric function $F_{\kappa}$ is defined as the unique holomorphic $W$-invariant function on $\mathbb{C}^{n}\times (\mathbb{R}^{n}+iU )$ ($U$ is a $W$-invariant neighborhood of 0) that satisfies the system of differential equations:
 \begin{eqnarray*} p(T_{e_1}, T_{e_2 }, \ldots, T_{e_n} ) F_{\kappa}(\lambda,\cdot)=p(\lambda)F_{\kappa}(\lambda, \cdot), \quad F_{\kappa}(\lambda, 0)=1  \end{eqnarray*}
for all $\lambda\in \mathbb{C}^{n}$ and all $W$-invariant polynomials $p$ on $\mathbb{R}^{n}$. For a fixed root system $R$, the hypergeometric
function $F_{\kappa}$ and the generalized Bessel function $J_{\kappa}$ satisfy the following rational limits
\begin{eqnarray}\label{lm1}J_{\kappa}(\lambda, x)=\lim_{m\rightarrow \infty} F_{\kappa}\left(m\lambda+\rho(\kappa), \frac{x}{m}\right). \end{eqnarray}
 Such limit transition was first obtained in \cite{BO} for  integer multiplicity function $\kappa$ and then later by de Jeu in \cite{de1}. It has been used to give an alternative proof for the positivity of the intertwining operator in \cite{RV} and to obtain an integral expression for the generalized Bessel functions of type $A_{n-1}$ in \cite{b1}.

 Recently, for the  root system of type $A_{n-1}$, the hypergeometric function $F_{\kappa}$ with a certain degenerate parameter was expressed explicitly by the Lauricella hypergeometric function $F_{D}$, see Theorem 2.2 and Theorem 3.1 in \cite{ST}. Recall that the Lauricella hypergeometric function $F_{D}$ is the analytic continuation of the series
 \begin{eqnarray*}
 &&F_{D}(a, b_{1}, \ldots, b_{n}, c; x_{1},\ldots, x_{n})\\&=&\sum_{m_{1},\ldots, m_{n}\ge 0  } \frac{(a)_{m_{1}+\cdots+m_{n}} (b_{1})_{m_{1}} \cdots (b_{n})_{m_{n}} }{(c)_{m_{1}+\cdots+m_{n}} }\frac{x_{1}^{m_{1}}\cdots x_{n}^{m_{n}}}{m_{1}!\cdots m_{n}!}
 \end{eqnarray*}
 where $a, b_{1}, \ldots, b_{n}, c$ are complex constants with $c\neq -1,-2,\ldots$ In the sequel, we denote the hyperplane $\mathbb{V}$ of  $\mathbb{R}^{n}$ given by
 \[\mathbb{V}=\{x\in \mathbb{R}^{n}: x_{1}+x_{2}+\cdots+x_{n}=0\}.\]

\begin{theorem}\label{st1} \cite{ST} Assume $\kappa\ge 0$, $\nu\in \mathbb{C}$ and $x\in \mathbb{V}$. Then the  Heckman-Opdam hypergeometric function for the root system of type $A_{n-1}$  can be written as
\begin{eqnarray*}
F_{\kappa}(\lambda(\nu)+\rho(\kappa),  x)=(y_{1}\cdots y_{n-1})^{-\frac{\nu}{n}}F_{D}(-\nu, \kappa, \ldots, \kappa, n\kappa; 1-y_{1}, \ldots, 1-y_{n-1}),
\end{eqnarray*}
where
$\lambda(\nu)=\left( -\frac{\nu}{n}, \cdots, -\frac{\nu}{n}, \frac{(n-1)\nu}{n} \right),$ $y_{j}=e^{x_{j}-x_{n}}, (1\le j\le n-1), y_{n}=e^{x_{n}}$ and $\rho(\kappa)=\frac{\kappa}{2}\sum_{\alpha\in R_{+}} \alpha.$
\end{theorem}
This result can be regarded as a generalization of the explicit expressions for the zonal spherical function on $SL(3,\mathbb{R})$ given in \cite{se} and  the Jack polynomials obtained in \cite{ta1}. However, for general $\lambda$, the classical hypergeometric representation for the Heckman-Opdam function $F_{\kappa}(\lambda, x)$ is much more complicated and  is still open.
\subsection{Humbert functions $\Phi_{2}^{(n)}$}
\label{sechum}

The Humbert function $\Phi_{2}^{(n)}$ of $n$ variables is defined by
\begin{eqnarray} \label{lm2}
&&\Phi_{2}^{(n)}[b_{1},\ldots, b_{n}; c; x_{1}, \ldots, x_{n}]
:=\sum_{m_{1}, \ldots, m_{n}=0}^{\infty}\frac{(b_{1})_{m_{1}}\cdots(b_{n})_{m_n}}{(c)_{m_{1}+\cdots+m_{n}}}
\frac{x_{1}^{m_{1}}}{m_{1}!} \cdots \frac{x_{n}^{m_{n}}}{m_{n}!}.\nonumber
\end{eqnarray}
It is  the confluent form of the Lauricella function $F_{D}$ and satisfies
 \begin{eqnarray*}
\Phi_{2}^{(n)}[b_{1},\ldots, b_{n}; c; x_{1}, \ldots, x_{n}]
=\lim_{|a|\rightarrow\infty} F_{D}\left[a, b_{1}, \ldots, b_{n}; c; \frac{x_{1}}{a}, \ldots, \frac{x_{n}}{a}\right],
\end{eqnarray*}
see \cite{HS} (Section 1.4, formula (10)).
When  $c-\sum_{j=1}^{n}b_{j}$  and each $b_{j}$,  $j=1, 2, \ldots, n$ are positive numbers, $\Phi_{2}^{(n)}$ has the following integral expression,
\begin{eqnarray}\label{pi1}&&\Phi_{2}^{(n)}(b_{1},
\ldots, b_{n}; c; x_{1}, \ldots, x_{n})\\&=&C_{b}^{ (c)}
\int_{T^{n}} e^{\sum_{j=1}^{n}x_{j}t_{j} } \left(1-\sum_{j=1}^{n}t_{j}\right)^{c-\sum_{j=1}^{m}b_{j}-1}\prod_{j=1}^{n}t_{j}^{b_{j}-1}dt_{1}\ldots dt_{n}\nonumber
\end{eqnarray}
where $C_{b}^{ (c)}=\frac{\Gamma{(c)}}{\Gamma(c-\sum_{j=1}^{n}b_{j})\prod_{j=1}^{n}\Gamma(b_{j})}$ and  $T^{n}$ is the open unit simplex in $\mathbb{R}^{n}$ given by
\begin{eqnarray*}T^{n}=\left\{(t_{1},\ldots, t_{n}): t_{j}>0, j=1,\ldots, n, \sum_{j=1}^{n}t_{j}<1\right\}.\end{eqnarray*}
We refer to \cite{cw,H} for more details on these functions.

\section{Generalized Bessel function of type $A_{n-1}$}

The limit relation (\ref{lm1}) of the integral kernels in the rational and  trigonometric setting together with (\ref{lm2}) leads to an explicit expression for the generalized Bessel function of type $A_{n-1}$.

\begin{theorem} Assume $\kappa\ge 0$, $\nu\in \mathbb{C}$ and $x\in \mathbb{V}\subset \mathbb{R}^{n}$. Then the  generalized Bessel function for the root system  $A_{n-1}$  is given by
\begin{eqnarray*}
J_{\kappa}(\lambda, x)&=&\Phi_{2}^{(n)}[\kappa, \ldots, \kappa, n\kappa; \nu x_{1}, \ldots, \nu x_{n} ]\\&=&e^{\nu x_{n}}\Phi_{2}^{(n-1)}[\kappa, \ldots, \kappa, n\kappa; \nu (x_{1}-x_{n}), \ldots, \nu (x_{n-1}-x_{n}) ]
\end{eqnarray*}
where
$\lambda=\left( -\frac{\nu}{n}, \ldots, -\frac{\nu}{n}, \frac{(n-1)\nu}{n} \right).$
\end{theorem}
\begin{proof}  For $x\in \mathbb{V}$, we adopt the rational limit relation (\ref{lm1}) to the explicit expression of the  Heckman-Opdam hypergeometric functions of Theorem \ref{st1}. This yields
\begin{eqnarray*}
&&J_{\kappa}(\lambda, x)\\&=&\lim_{m\rightarrow \infty} F\left(m\lambda+\rho(\kappa), \frac{x}{m}\right)\\
&=& \lim_{m\rightarrow \infty} \left(e^{\frac{x_{1}-x_{n}}{m}}\cdots e^{\frac{x_{n-1}-x_{n}}{m}}\right)^{-\frac{m\nu}{n}}
\\ &&\times F_{D}\left(-m\nu, \kappa, \ldots, \kappa, n\kappa; 1-e^{\frac{x_{1}-x_{n}}{m}}, \ldots, 1-e^{\frac{x_{n-1}-x_{n}}{m}}\right)\\
&=&(y_{1}\cdots y_{n-1})^{-\frac{\nu}{n}}\lim_{m\rightarrow \infty}F_{D}\left(-m\nu, \kappa, \ldots, \kappa, n\kappa; 1-e^{\frac{x_{1}-x_{n}}{m}}, \ldots, 1-e^{\frac{x_{n-1}-x_{n}}{m}}\right).
\end{eqnarray*}
By the limit relation (\ref{lm2}) and  the fact \[ \lim_{m\rightarrow \infty} \frac{1-e^{(x_{j}-x_{n})/m}}{(x_{n}-x_{j})/m} =1,\]
we obtain
\begin{eqnarray*}
&&J_{\kappa}(\lambda, x)\\
&=&(y_{1}\cdots y_{n-1})^{-\frac{\nu}{n}}\lim_{m\rightarrow \infty}F_{D}\left(-m\nu, \kappa, \ldots, \kappa, n\kappa; 1-e^{\frac{x_{1}-x_{n}}{m}}, \ldots, 1-e^{\frac{x_{n-1}-x_{n}}{m}}\right)\\
&=& e^{-\frac{\nu}{n}\sum_{j=1}^{n-1}(x_{j}-x_{n})} \Phi_{2}^{(n-1)}[\kappa, \ldots, \kappa, n\kappa; \nu(x_{1}-x_{n}), \ldots, \nu(x_{n-1}-x_{n}) ].
\end{eqnarray*}
Here $\Phi_{2}^{(n-1)}$ is the second class of Humbert functions, see Section \ref{sechum}.
 Note that as $x\in \mathbb{V}$, we have $\sum_{j=1}^{n-1}(x_{j}-x_{n})=-nx_{n}$. Therefore, for  $x\in \mathbb{V}$,
\begin{eqnarray*}
J_{\kappa}(\lambda, x)
&=& e^{\nu x_{n}} \Phi_{2}^{(n-1)}[\kappa, \ldots, \kappa, n\kappa; \nu(x_{1}-x_{n}), \ldots, \nu(x_{n-1}-x_{n}) ]\\
&=& \Phi_{2}^{(n)}[\kappa, \ldots, \kappa, n\kappa; \nu x_{1}, \ldots, \nu x_{n-1}, \nu x_{n} ],
\end{eqnarray*}
where the last identity is obtained by the Laplace transform of $\Phi_{2}^{(n)}$, see \cite{dl, DD}.
\end{proof}
If we take $\nu=1$, we get the following corollary.
\begin{corollary}\label{co1}
 For $\lambda=\left(-\frac{1}{n}, \ldots, -\frac{1}{n}, \frac{n-1}{n}\right)$, $x\in \mathbb{V}$, the generalized Bessel function $J_{\kappa}(\lambda, x)$ of type $A_{n-1}$ is given by
 \begin{eqnarray*}
J_{\kappa}(\lambda, x)
&=& e^{ x_{n}} \Phi_{2}^{(n-1)}[\kappa, \ldots, \kappa, n\kappa; x_{1}-x_{n}, \ldots, x_{n-1}-x_{n} ]\\
&=&c_{\kappa}\int_{T^{n-1}}e^{\sum_{j=1}^{n}x_{j} t_{j} } \prod_{j=1}^{n}t_{j}^{\kappa-1}dt_{1}\ldots dt_{n-1},
\end{eqnarray*}
where $t_{n}=1-\sum_{j=1}^{n-1}t_{j}$ and $c_{\kappa}=\Gamma (n\kappa)/(\Gamma(\kappa)^{n})$.
\end{corollary}

Alternatively, the generalized Bessel function is defined as the symmetric analogue of the Dunkl kernel,
 \begin{eqnarray*}
J_{\kappa}(\lambda, x)&=&\frac{1}{n!}\sum_{\sigma\in S_{n}}E_{\kappa}(x, \lambda\sigma)
=\frac{1}{n!}\sum_{\sigma\in S_{n}}V_{\kappa}\left[e^{\langle \cdot, \lambda\sigma \rangle}\right](x).
\end{eqnarray*}
Furthermore, when $\lambda=\left(-\frac{1}{n}, \ldots, -\frac{1}{n}, \frac{n-1}{n} \right)$, the condition $x\in \mathbb{V}$ yields $\langle x, \lambda\rangle= x_{n}$ and the exponential becomes \[e^{\langle x,  \lambda\rangle}=e^{x_{n}}=e^{\langle x,  e_{n}\rangle},\]
which only depends on the component $x_{n}$, here $e_{n}=(0, 0, \ldots, 1)$. Combining this with Corollary \ref{co1}, for any $x\in \mathbb{V}$, we have
 \begin{eqnarray*}
J_{\kappa}(\lambda, x)&=&\frac{1}{n!}\sum_{\sigma\in S_{n}}V_{\kappa}\left[e^{\langle \cdot, e_{n}\sigma\rangle }\right](x)\\
&=&c_{\kappa}\int_{T^{n-1}}e^{\sum_{j=1}^{n}x_{j} t_{j} } \prod_{j=1}^{n}t_{j}^{\kappa-1}dt_{1}\ldots dt_{n-1}.
\end{eqnarray*}
Moreover, since the intertwining operator is homogenous, i.e. $V_{\kappa}: \mathcal{P}_{m}^{n}\rightarrow \mathcal{P}_{m}^{n}$,  and the generalized Bessel functions are analytic,  we have
 \begin{eqnarray} \label{ip}
 &&V_{\kappa}\left[\sum_{\sigma\in S_{n}} \langle \cdot, e_{n}\sigma\rangle^{m} \right](x)
=n!c_{\kappa}\int_{T^{n-1}}\left(\sum_{j=1}^{n}x_{j} t_{j} \right)^{m} \prod_{j=1}^{n}t_{j}^{\kappa-1}dt_{1}\ldots dt_{n-1}
\end{eqnarray}
for $x\in \mathbb{V}$. By analytic continuation, it is  seen that the integral expression also works for all $x \in \mathbb{R}^{n}$.

The above formula (\ref{ip}) further leads to an explicit expression of the intertwining operator for general functions by a limit discussion.
\begin{theorem}\label{si1} For $x\in \mathbb{R}^{n}$ and a function $f(x_{j})$ in a single component,  define a $S_{n}$-invariant function by
$F(x)=\sum_{\sigma\in S_{n}}f(x_{j}\sigma )$.
 Then the intertwining operator acting on $F(x)$ is given by
 \begin{eqnarray*}
V_{\kappa}F(x) =n!c_{\kappa}\int_{T^{n-1}}f(x_{1}t_{1}+x_{2}t_{2}+\ldots+x_{n}t_{n})\prod_{j=1}^{n}t_{j}^{\kappa-1}dt_{1}\ldots dt_{n-1}.
\end{eqnarray*}
\end{theorem}
\begin{remark} This can also be verified by checking the intertwining relations of (\ref{ir1}) directly  in a similar way  as in \cite{Xu}.
\end{remark}

\begin{corollary} For $1\le \ell \le n$, $x\in \mathbb{R}^{n}$ and $e_{\ell}=e_{n}(\ell,n)$, the generalized Bessel function $J_{\kappa}(e_{\ell}, x ) $ is given by
 \begin{eqnarray*}
 J_{\kappa}(e_{\ell}, x)=c_{\kappa}\int_{T^{n-1}}e^{\langle x, t\rangle}\prod_{j=1}^{n}t_{j}^{\kappa-1}dt_{1}\ldots dt_{n-1}.
\end{eqnarray*}
\end{corollary}
\begin{proof} Since
\begin{eqnarray*}J_{\kappa}(e_{\ell}, x)=V_{\kappa}\left[\frac{1}{n!}\sum_{\sigma\in S_{n}}e^{\langle \cdot, e_{\ell}\sigma\rangle}\right](x),
\end{eqnarray*}
we put $f(x_{\ell})=\frac{1}{n!}e^{x_{\ell}}= \frac{1}{n!}e^{\langle x, e_{\ell}\rangle}$ in Theorem \ref{si1} and then obtain the formula.
\end{proof}

\begin{remark}The same formula was  obtained in  \cite{Xu}, Corollary 2.4.
\end{remark}

\section{Dunkl kernel and intertwining operator of type $A_{n-1}$}
It was routine to use the shift principle of \cite{OP} to derive the Dunkl kernel from the generalized Bessel function, see e.g. \cite{ba1}. For our purpose, there exists a simpler way to achieve this goal. However,  we still start  from computing for the root system $A_{2}$ using the shift principle to show how it works. Denote by $W(\lambda)$ the alternating polynomial  associated to $A_{2}$ \[W(\lambda)=(\lambda_{1}-\lambda_{2})(\lambda_{1}-\lambda_{3} ) (\lambda_{2}-\lambda_{3}) .\] 
\begin{theorem}\label{t2} For the root system $A_{2}$,  $x\in \mathbb{R}^{3}$ and $e_{3}=(0, 0, 1)$, the Dunkl kernel can be expressed as
 \begin{eqnarray*}
E_{\kappa}(x, e_{3})=V_{\kappa}\left(e^{\langle \cdot, e_{3}\rangle}\right)(x)
&=&e^{x_{3}}\Phi_{2}^{(2)}[\kappa,\kappa,3\kappa+1; x_{1}-x_{3}, x_{2}-x_{3}]
 \\&=&\Phi_{2}^{(3)}[\kappa,\kappa,\kappa,3\kappa+1; x_{1}, x_{2}, x_{3}].
\end{eqnarray*}
\end{theorem}
\begin{proof} Recall that the shift principle
 implies (see Proposition 1.4 in \cite{du})
 \begin{eqnarray} \label{df}\sum_{\sigma\in S_{3}}det(\sigma)E_{\kappa}(x\sigma, \lambda )=\gamma_{\kappa}W(x)W(\lambda)J_{\kappa+1}(x,\lambda)  \end{eqnarray}
 where $\gamma_{\kappa}$ is a normalizing constant which will not be explicitly used here.
Combining  (\ref{df}) with (\ref{bes1}), we have
 \begin{eqnarray} \label{b2}&& E_{\kappa}(x,\lambda)+ E_{\kappa}(x,\lambda\sigma)+ E_{\kappa}(x,\lambda\sigma^{2})\\&=&\frac{1}{2}\left(\gamma_{\kappa}W(\lambda)W(x)J_{\kappa+1}(x,\lambda) +6J_{\kappa}(x, \lambda)\right) \nonumber
 \end{eqnarray}
where $\sigma=(1,3)(1,2)$, here $(i,j)$ is the transposition exchanging  the $i$th and $j$th coordinates of $x\in \mathbb{R}^{n}$. This relation (\ref{b2}) has been used to derive an integral expression for the Dunkl kernel of type $A_{2}$ in \cite{ba1}.

Now, we act with the Dunkl operator $D_{3}^{(x)}$ on both sides of (\ref{b2}) with $\lambda=e_{3}=(0, 0, 1)$. Using the relations
\begin{eqnarray*}D_{j}^{(x)}E_{\kappa}(x,\lambda)=\lambda_{j}E_{\kappa}(x,\lambda), \quad j=1,2, 3, \end{eqnarray*} this yields
 \begin{eqnarray*}
 E_{\kappa}(x, e_{3})&=&D_{3}E_{\kappa}(x, e_{3})\\&=&D_{3} (E_{\kappa}(x, e_{3})+ E_{\kappa}(x, e_{3}\sigma)+ E_{\kappa}(x, e_{3}\sigma^{2}))\\&=&\frac{1}{2}D_{3}\left(\gamma_{\kappa}W(e_{3})W(x)J_{\kappa+1}(x, e_{3}) +6J_{\kappa}(x, e_{3})\right)\\&=&3\partial_{3}J_{\kappa}(x, e_{3})\\
 &=& 3\partial_{3} c_{\kappa} \int_{T^{2}}e^{(x_{1}t_{1}+x_{2}t_{2}+x_{3}t_{3})}(t_{1}t_{2}t_{3})^{\kappa-1}dt_{1}dt_{2}
 \\&=&3c_{\kappa}\int_{T^{2}}e^{(x_{1}t_{1}+x_{2}t_{2}+x_{3}t_{3})}t_{3}(t_{1}t_{2}t_{3})^{\kappa-1}dt_{1}dt_{2}
 \\&=&e^{x_{3}}\Phi_{2}^{(2)}[\kappa,\kappa,3\kappa+1; x_{1}-x_{3}, x_{2}-x_{3}]
 \\&=&\Phi_{2}^{(3)}[\kappa,\kappa,\kappa,3\kappa+1; x_{1}, x_{2}, x_{3}].
  \end{eqnarray*}
Here the third identity holds by the fact  $W(e_{3})=0$ and that $J_{\kappa}(x, e_{3})$ is $S_{3}$-invariant in the variable $x$.

\end{proof}

In the following, we consider the general symmetric group $S_{n}$. Recalling the representation (\ref{bes1}), for $1\le \ell\le n$,  the generalized Bessel function of type $A_{n-1}$ can also be expressed as
\begin{eqnarray} \label{bes2} J_{\kappa}(e_{\ell}, x)=J_{\kappa}(e_{1}, x)=\frac{1}{n}\sum_{j=1}^{n}E_{\kappa}(e_{1}, x(1,j)).
\end{eqnarray}
Hence, $E_{\kappa}(x, e_{\ell})$ can be obtained by acting with $D_{m}^{(x)}$ on both sides of (\ref{bes2}) using the relations
\[D_{j}^{(x)}E_{\kappa}(x,\lambda)=\lambda_{j}E_{\kappa}(x,\lambda), \quad j=1,2,\ldots, n.\]
Similar as Theorem \ref{t2}, we then have a result for $A_{n-1}$,
\begin{theorem} For root  system $A_{n-1}$, $x\in \mathbb{R}^{n}$, the Dunkl kernel admits
 \begin{eqnarray*}
E_{\kappa}(x, e_{\ell})&=&V_{\kappa}\left(e^{\langle \cdot, e_{\ell}\rangle}\right)(x)\\&=&e^{x_{n}}\Phi_{2}^{(n-1)}(\kappa, \ldots, \underbrace{\kappa+1}_{\ell}, \ldots \kappa; n\kappa+1; x_{1}-x_{n}, \ldots, x_{n-1}-x_{n})
\\&=&nc_{\kappa}\int_{T^{n-1}}e^{\sum_{j=1}^{n}x_{j} t_{j} } t_{\ell}\prod_{j=1}^{n}t_{j}^{\kappa-1}dt_{1}\ldots dt_{n-1},
\end{eqnarray*}
where $t_{n}=1-\sum_{j=1}^{n-1}t_{j}$ and $c_{\kappa}=\Gamma (n\kappa)/(\Gamma(\kappa)^{n})$.
\end{theorem}
\begin{remark} This expression is first given in \cite{Xu}, Corollary 2.4.
\end{remark}
\begin{remark} Our results depend heavily on Theorem 2.1 of \cite{ST}, which only works for degenerated parameters. It is not clear to us how to obtain a full classical hypergeometric expression for the Heckman-Opdam function $F_{\kappa}(\lambda, x)$  at present. Hence, this approach can not lead to a full expression for the intertwining operator as  in the dihedral case. On the other side, these results both for the dihedral  and symmetric groups show that the Dunkl kernel at some special lines (corresponding to the degenerated parameters) behaves much  simpler  than in other points. For general root systems and those special lines, the Dunkl kernel and the intertwining operator  may have a close relationship with the Humbert functions as well.
\end{remark}
Since the intertwining operator maps polynomials of degree $m$ to polynomials of the same degree, we have
\[V_{\kappa}(x_{\ell}^{m} )=nc_{\kappa}\int_{T^{n-1}}(x_{1}t_{1}+x_{2}t_{2}+\cdots+x_{n}t_{n})^{m}t_{\ell} \prod_{j=1}^{n}t_{j}^{\kappa-1}dt_{1}dt_{2}\ldots dt_{n-1}  .\]
This  leads to an explicit expression for the intertwining operator when the  function is of a single component, which was obtained recently by Xu in \cite{Xu} Theorem 2.1.
\begin{theorem}
Let $f: \mathbb{R}\rightarrow \mathbb{R}$. For $1\le \ell\le n,$ define $F(x_{1}, x_{2}, \ldots, x_{n})=f(x_{\ell})$. Then the intertwining operator acting on $F(x)$ is given by
\[V_{\kappa}F(x)=c_{\kappa}^{(n)}\int_{T^{n-1}}f(x_{1}t_{1}+x_{2}t_{2}+\cdots+x_{n}t_{n} )t_{\ell}\prod_{j=1}t_{j}^{\kappa-1} dt_{1}dt_{2}\ldots dt_{n-1}\]
where $c_{\kappa}^{(n)}=nc_{\kappa}=\Gamma(n\kappa+1)/\left(\kappa \Gamma(\kappa)^{n}\right)$.
\end{theorem}

\section{Conclusion}
In this note, we have expressed the generalized Bessel function and Dunkl kernel of type $A_{n-1}$ in terms of the Humbert function $\Phi_{2}^{(n)}$, with one variable fixed. A new proof of Xu's integral formula for the intertwining operator was developed by these formulas. The same approach will also lead to  explicit expressions for the trigonometric Dunkl intertwining operator associated to the dihedral and symmetric groups.



\end{document}